\documentclass[11pt]{amsart}

\usepackage[T1]{fontenc}
\usepackage{lmodern}
\usepackage{microtype}
\usepackage{amsmath,amssymb,amsthm,mathtools}
\usepackage{xcolor}
\usepackage[a4paper,margin=31mm]{geometry}
\usepackage[colorlinks=true,linkcolor=blue!55!black,citecolor=blue!55!black,urlcolor=blue!55!black]{hyperref}
\usepackage[nameinlink,noabbrev]{cleveref}
\hypersetup{pdftitle={The Entire-Graph Anisotropic Bernstein Theorem for Integrands C2-Close to Area}}

\newtheorem{theorem}{Theorem}[section]
\newtheorem{lemma}[theorem]{Lemma}
\newtheorem{proposition}[theorem]{Proposition}

\newcommand{\Hn}{\mathcal{H}^{n}}
\newcommand{\Aphi}{\mathcal{A}_{\Phi}}

\newcommand{\R}{\mathbb R}
\newcommand{\Sph}{\mathbb S}
\newcommand{\Haus}{\mathcal H}
\newcommand{\Per}{P}
\newcommand{\F}{\mathcal F}
\newcommand{\supp}{\operatorname{spt}}

\newcommand{\osc}{\operatorname{osc}}

\newcommand{\diver}{\operatorname{div}}

\newcommand{\Id}{\mathrm{Id}}
\newcommand{\Tr}{\operatorname{Tr}}
\newcommand{\Cyl}{\mathcal C}
\newcommand{\Mass}{\mathbf M}
\newcommand{\llbracketbracket}[1]{\mathopen{[\![}#1\mathclose{]\!]}}

\title[$C^2$-near-area anisotropic Bernstein theorem]
{A $C^2$-Perturbative Bernstein Theorem for Anisotropic Entire Minimal Graphs}
\author{}
\date{}
\author{Lu Chen}
\address[Lu Chen]{Key Laboratory of Algebraic Lie Theory and Analysis of Ministry of Education, School of Mathematics and Statistics, Beijing Institute of Technology, Beijing
100081, PR China;  Tangshan Research Institute, Beijing Institute of Technology, Tangshan 063000, PR China}
\email{chenlu5818804@163.com}

\author{Jiali Lan}
\address[Jiali Lan]{Key Laboratory of Algebraic Lie Theory and Analysis of Ministry of Education, School of Mathematics and Statistics, Beijing Institute of Technology, Beijing
100081, PR China}
\email{17636268505@163.com}

\address{}

\thanks{The first author was partly supported by the  National Natural Science Foundation of China (No. 12622108),  a grant from Beijing Institute of Technology (No.2022CX01002).  }
\begin{document}
\begin{abstract}

We prove a  Bernstein theorem for $\Phi$-anisotropic minimal hypersurfaces in dimensions $1\leq n\leq 7$ that the only entire smooth solutions $u$ of $\Phi$-anisotropic minimal hypersurfaces  equation are affine functions provided the anisotropic area functional integrand $\Phi$ is sufficiently $C^{2}$-close to the Euclidean area integrand. This settles the $C^2$ entire-graph version of anisotropic Bernstein problem posed by Mooney and Yang \cite{MooneyYang2024}, and the proof uses a compactness-rigidity argument combined with Figalli's regularity theorem established in \cite{Figalli2017}.

\end{abstract}

\maketitle

\section{Introduction}\label{sec:introduction}

In this paper we study entire graphical solutions of anisotropic minimal hypersurfaces in $\R^{n+1}$ as critical points of parametric elliptic functionals.  For an oriented hypersurface $\Sigma\subset\R^{n+1}$ with unit normal $\nu_{\Sigma}$, its
anisotropic area is
\begin{equation}\label{eq:anisotropic-area}
    \Aphi(\Sigma):=\int_{\Sigma}\Phi(\nu_{\Sigma})\,d\Hn,
\end{equation}
where $  \Phi\colon \R^{n+1}\setminus\{0\}\longrightarrow (0,\infty)$
is convex, positively one-homogeneous, and sufficiently regular away from the origin.  The density $\Phi$ geometrically represents a Minkowski norm on $\mathbb{R}^{n+1}$ and physically models the direction-dependent surface tension of an interface. In materials science, this functional arises naturally to describe the surface free energy of crystalline solids, where different crystallographic planes exhibit distinct packing densities and atomic bonds, leading to the celebrated Wulff construction for equilibrium crystal shapes.


If $\Sigma$ is the graph of $u\colon\R^{n}\to\R$, with the upward orientation, set
\begin{equation}\label{eq:graph-integrand}
    \varphi(p):=\Phi(-p,1),\qquad p\in\R^{n}.
\end{equation}
Then the anisotropic area of the graph over a bounded set $\Omega\subset\R^{n}$ is
$\int_{\Omega}\varphi(Du)\,dx$, and the Euler-Lagrange equation is
\begin{equation}\label{eq:anisotropic-graph-equation}
    \operatorname{div}\bigl(D\varphi(Du)\bigr)=0
    \qquad\text{in }\R^{n}.
\end{equation}

The anisotropic Bernstein problem asks  whether every entire solution of \eqref{eq:anisotropic-graph-equation} is affine. In particular, when $\Phi(\xi)=|\xi|$, it is known through  pioneering works   of Bernstein, Fleming \cite{Fleming1962}, De Giorgi \cite{DeGiorgi1965}, Almgren \cite{Almgren1966} and Simons \cite{Simons1968} that  the only entire smooth graphical solutions of the classical graphical minimal hypersurface equation on $\R^n$ are linear functions, provided that  $n \leq 7$, and the result fails when $n>7$ by the counterexample of  Bombieri-De Giorgi-Giusti in \cite{BDGG1969}.
For   general uniformly elliptic integrands,   Jenkins \cite{Jenkins1961} proved that all entire smooth $\Aphi$--minimal graphs are linear in dimension $n=2$, and Simon  \cite{Simon1977} did so in dimension $n=3$.  By constructing
nonflat minimizers  for certain $\Phi$ away from the area integrand, Mooney-Yang \cite{MooneyYang2024} showed that the anisotropic Bernstein theorem for entire smooth graphical solutions can only hold up to $n=3$.
We also mention that in a recent work \cite{ChenLanLiu2025}, the authors established the solvability of the anisotropic Dirichlet problem under a natural nonnegative anisotropic mean curvature condition,  extending the classical Jenkins-Serrin theorem \cite{JS} to the anisotropic case.

For integrands close to the Euclidean area integrand, Simon
\cite{Simon1977} proved that the Bernstein theorem holds for entire
graphs in dimensions $n\leq7$, provided that the integrand of the
parametric functional \eqref{eq:anisotropic-area} is sufficiently close
to the area integrand in $C^3(\Sph^n)$. For complete, two-sided, stable
anisotropic minimal hypersurfaces without boundary, flatness holds in
dimension $n=2$ under $C^2$ closeness by Lin \cite{Lin1990}, and in
dimension $n=3$ under $C^4$ closeness by Chodosh--Li
\cite[Corollary~1.5]{CL2023}. The latter authors also observed that their
flatness conclusion holds under the weaker assumption of
$C^{2,\alpha}$ closeness \cite[Remark~1.6]{CL2023}. Mooney and Yang
\cite[Section~6.4]{MooneyYang2024} asked whether the closeness assumptions
in such results could be weakened to $C^2$. More recently, Du and Yang
\cite{DuYang2024} established a Bernstein theorem in every dimension
under $C^3$ closeness and an additional quantitative growth restriction
on $|Du|$. These results leave open the entire-graph problem under
$C^2$ closeness in dimensions $n\leq7$, without additional growth
assumptions.


The relevance of the $C^2$ topology is reflected in the variational
structure. Indeed, if $\phi=\Phi|{\Sph^n}$, $\nu\in\Sph^n$, and
$\tau\perp\nu$, positive one-homogeneity gives
\begin{equation}\label{eq:spherical-hessian}
D^2\Phi(\nu)[\tau,\tau]
=\nabla^2_{\Sph^n}\phi(\nu)[\tau,\tau]
+\phi(\nu)|\tau|^2.
\end{equation}
Thus $C^2$ closeness on the sphere controls the integrand and the
coefficients entering its first and second variations, including the
parametric ellipticity. It does not provide bounds for higher
derivatives, even when $\Phi$ is smooth. The question is therefore
whether the Bernstein theorem can be proved using estimates that
remain uniform under $C^2$ perturbations, without bounds on higher
derivatives of the integrand. We give an affirmative answer for
entire graphs.


\begin{theorem}[Entire-graph $C^2$ Bernstein theorem]\label{thm:main}
 For every $1\le n\le 7$, let $u$ be an entire smooth solution to the anisotropic minimal surface equation
\begin{equation}\label{eq:graphpde}
 \diver\bigl(D\varphi(Du)\bigr)=0\quad\text{in }\R^n,
\end{equation}
where $\varphi$ is defined by \eqref{eq:graph-integrand}. If there exists $\delta_n>0$ such that
\begin{equation}\label{eq:c2close}
\|\Phi-1\|_{C^2(\Sph^n)}<\delta_n,
\end{equation}
then $u$ is affine.
\end{theorem}

The proof is based on a compactness--rigidity scheme  that depends only on the $C^{2}$ structure of the integrand. If the theorem
fails,    there exists a sequence of non-affine entire solutions $u_j$ whose integrands $\Phi_j$ converge to the Euclidean area integrand $|\cdot|$ in $C^{2}(\mathbb{S}^n)$.  We will  shows that each subgraph of $u_j$ is a global $\Phi_j$-minimizer and that its oriented boundary current is locally absolutely minimizing. Moreover, the subgraphs satisfy a one-sided translation monotonicity, which is preserved under local \(L^{1}\) convergence. Nonflatness is then characterized  by a quantitative Gauss-image gap: if a non-flat counterexample exists, it must contain two points with a scale-invariant separation of their normals. After normalizing by the distance between these two points, compactness yields a limiting Euclidean perimeter minimizer. A blow-down of this global minimizer yields a nontrivial directed minimizing cone, which must be a halfspace for $n\leq 7$ by the classification of directed minimizing cones. The Euclidean monotonicity formula then  transfers the halfspace rigidity from the blow-down to the original limiting minimizer, showing that  the original limiting minimizer is itself a halfspace. Once the limit is a multiplicity-one plane, Figalli's regularity theorem improves the convergence to $C^{1,\alpha}$ graphical convergence and uniform normal convergence, contradicting the Gauss-image gap. The argument thus splits into three components: variational compactness, directed Euclidean rigidity, and low-regularity improvement of flatness.

We finally emphasize the scope of the result.  Mooney and Yang's discussion in
\cite[Section~6.4]{MooneyYang2024} includes broader $C^{2}$-closeness questions for stable
critical hypersurfaces.  Theorem~\ref{thm:main} settles the entire-graph branch, where
the calibration and directedness of a subgraph provide global structure not available for a
general stable critical hypersurface.

This paper is organized as follows.  In
Section~\ref{sec:variational-compactness}, we transform  the
entire-graph equation into a variational problem, establish the quantitative
Gauss-image gap, and develop compactness for minimizers with a sequence of
anisotropic integrands.  In Section~\ref{sec:directed-rigidity}, we  obtain  the classification of
directed Euclidean global perimeter minimizers in dimensions  $n\leq7$.  In
Section~\ref{sec:flatness-main-proof},  we apply Figalli's regularity theorem to obtain the low-regularity flat compactness needed near a multiplicity-one plane, and complete the proof of Theorem~\ref{thm:main}.

\section{Variational Method and Compactness}
\label{sec:variational-compactness}

This section develops the variational and compactness argument.  We derive uniform ellipticity from the \( C^2 \) assumption, and then obtain the Gauss-image gap lemma. We also prove a compactness theorem for a sequence of $\Phi_j$-minimizers. We first introduce the uniform ellipticity of the integrand  $\Phi$.


Let $U\subset\mathbb{R}^{n+1}$ be an open set. The  total variation of a function $u\in BV(U)$ with respect to  $\Phi$ is given by
\[\int_{U}|Du|_{\Phi}=\sup\left\{\int_{U}u\mathrm{div}\sigma dx: \sigma\in C_0^1(U,\mathbb{R}^{n+1}),\  \Phi^o(\sigma)\leq 1\right\},\]
where $\Phi^o(z):=\sup_{\xi}\frac{\left<\xi,z\right>}{\Phi(\xi)}$ is the dual metric of $\Phi$.
This yields the  generalized  definition of perimeter of a set $E$ with respect to $\Phi$:
\[P_{\Phi}(E,U)=\int_{U}|D\chi_{E}|_{\Phi}=\sup\left\{\int_{E}\mathrm{div}\sigma dx: \sigma\in C_0^1(U,\mathbb{R}^{n+1}),\ \Phi^o(\sigma)\leq 1\right\},\]
 which can be rewritten as
\[P_{\Phi}(E,U)=\int_{U\cap\partial^* E}\Phi(\nu_E)d\mathcal{H}^{n},\]
  where $\partial^* E$ is the reduced boundary of $E$ and $\nu_E$ is the outer normal to $E$ (see \cite{AB}). We call $E$ a global $\Phi$--minimizer if $P_{\Phi}(E,U)\leq P_{\Phi}(F,U)$ whenever $U\Subset\mathbb{R}^{n+1}$ and $E \triangle F\Subset U$.

   If $\phi=\Phi|_{\Sph^n}$, $\nu\in\Sph^n$,
and $\tau\perp\nu$, positive one-homogeneity gives
 \begin{equation}\label{eq:hessian-sphere-early}
 D^2\Phi(\nu)[\tau,\tau]
 =\nabla^2_{\Sph^n}\phi(\nu)[\tau,\tau]
   +\phi(\nu)|\tau|^2.
\end{equation}
Thus, after requiring $\delta_n<1/4$, one has
\[
 \tfrac12|\tau|^2\le D^2\Phi(\nu)[\tau,\tau]
 \le\tfrac32|\tau|^2,
 \qquad \tfrac34\le\Phi(\nu)\le\tfrac54,
\]
Therefore, the integrand $\Phi$ is uniformly elliptic, and its restriction to the sphere has a uniform $C^{1,1}$ bound.


We now introduce some notation. Fix the Euclidean volume form $dV$.
For an oriented unit simple $n$-vector $\xi$, let $\nu_\xi$ be the
associated unit normal determined by
$
(\iota_{\nu_\xi}dV)(\xi)=1.
$
This fixes all orientation signs below. For an oriented rectifiable
$n$-current $T$ and an open set $U$, set
\[
\F_\Phi(T;U)
:=
\int_U\Phi(\nu_{\vec T})\,d\|T\|.
\]
We omit $U$ when $T$ has finite mass and the integral is over all of
$\R^{n+1}$. For a set $E$ of locally finite perimeter, we orient
$\llbracketbracket{E}$ by $dV$. Its boundary current
$T=\partial\llbracketbracket{E}$ then satisfies
\[
\nu_{\vec T}=\nu_E
\qquad\|T\|\text{-almost everywhere}.
\]


\begin{lemma}\label{lem:calibration}
Let
$u\in C^2(\R^n)$ be the solution of \eqref{eq:graphpde}.  Then the subgraph of $u$
\[
 E_u:=\{(x,t)\in\R^n\times\R:t<u(x)\}
\]
is a global $\Phi$-minimizer.  Moreover,
\begin{equation}\label{eq:directed-original}
 E_u-se_{n+1}\subset E_u\qquad(s>0).
\end{equation}
In fact, the oriented boundary current
$T_u=\partial\llbracketbracket{E_u}$ is locally $\Phi$--minimizer:
for all rectifiable $n$-currents $X$ with $\partial X=0$ and $\supp X\Subset U\Subset\R^{n+1}$, we have
 \[\F_\Phi(T_u;U)\le \F_\Phi(T_u+X;U).\]
\end{lemma}

\begin{proof}
The outer unit normal to the subgraph is
\[
 \nu(x)=\frac{(-Du(x),1)}{\sqrt{1+|Du(x)|^2}}.
\]
The graph area formula and one-homogeneity show that its anisotropic area
over a bounded base domain $\Omega$ is
\[
 \int_\Omega \Phi(-Du,1)\,dx=\int_\Omega\varphi(Du)\,dx.
\]
Since $D\Phi$ is zero-homogeneous, the vector field
\[
 Z(x,t):=D\Phi(-Du(x),1)=D\Phi(\nu(x))
\]
is well defined and independent of $t$.  Since
$\varphi(p)=\Phi(-p,1)$, we can rewritten \eqref{eq:graphpde} as
\[
 \diver_{\R^{n+1}}Z=0.
\]

Extending $\Phi$ to the origin by $\Phi(0)=0$, convexity and
one-homogeneity imply, for $\zeta\neq0$ and $\eta\in\R^{n+1}$, that
\[
 \Phi(\eta)
 \geq\Phi(\zeta)+D\Phi(\zeta)\cdot(\eta-\zeta)
 =D\Phi(\zeta)\cdot\eta.
\]
Therefore,
\begin{equation}\label{eq:calibration-ineq}
 D\Phi(\nu)\cdot\eta\leq\Phi(\eta),
 \qquad D\Phi(\nu)\cdot\nu=\Phi(\nu).
\end{equation}
Let $dV$ be the fixed Euclidean volume form and set
$\omega=\iota_ZdV$. Since
$d(\iota_ZdV)=(\diver Z)dV$, the $n$-form $\omega$ is closed.
For every oriented unit $n$-vector $\xi$,
\eqref{eq:calibration-ineq} gives
\begin{equation}\label{e2.1}
 \omega(\xi)=Z\cdot\nu_\xi\leq\Phi(\nu_\xi),
\end{equation}
with equality on $T_u$.

Since $d\omega=0$, its standard mollifications
$\omega_\varepsilon=\rho_\varepsilon*\omega$ are smooth, closed,
and uniformly bounded. The radial homotopy operator applied to
$\omega_\varepsilon$ gives the smooth $(n-1)$-form
\begin{equation}\label{eq:radial-primitive}
 (\alpha_\varepsilon)_x(v_1,\ldots,v_{n-1})
 =\int_0^1t^{n-1}
   (\omega_\varepsilon)_{tx}(x,v_1,\ldots,v_{n-1})\,dt,
\end{equation}
which satisfies
\begin{equation}\label{eq:radial-bound}
 d\alpha_\varepsilon=\omega_\varepsilon,
 \qquad |\alpha_\varepsilon(x)|\leq C|x|,
\end{equation}
with $C$ independent of $\varepsilon$.

Let $X$ be any rectifiable $n$-current with
$\mathbf M(X)<\infty$ and $\partial X=0$. Choose
$\chi_R\in C_c^\infty(B_{2R})$ such that
\[
 0\leq\chi_R\leq1,\qquad
 \chi_R=1\ \text{on }B_R,\qquad
 |D\chi_R|\leq C/R.
\]
Since $\chi_R\alpha_\varepsilon$ is smooth and compactly supported,
$\partial X=0$ gives
\begin{equation}\label{eq:calibration-cutoff}
 0=\partial X(\chi_R\alpha_\varepsilon)
  =X(d\chi_R\wedge\alpha_\varepsilon)
   +X(\chi_R\omega_\varepsilon).
\end{equation}
By \eqref{eq:radial-bound},
\[
 \left|X(d\chi_R\wedge\alpha_\varepsilon)\right|
 \leq C\|X\|(B_{2R}\setminus B_R)
 \longrightarrow0
 \qquad\text{as }R\to\infty.
\]
Since $\omega_\varepsilon$ is bounded and $\mathbf M(X)<\infty$,
dominated convergence yields $X(\omega_\varepsilon)=0$.
Letting $\varepsilon\to0$ and using the dominated
convergence again, we obtain
\[
 X(\omega)=0.
\]

Now let $U$ and $X$ be as in the statement. The compact support
of $X$ yields that $\mathbf M(X)<\infty$. Set
$S=T_u\llcorner U$, since
$\supp X\Subset U$, \eqref{e2.1} gives
\[
 \F_\Phi(T_u+X;U)
 =\F_\Phi(S+X)
 \geq(S+X)(\omega)
 =S(\omega)
 =\F_\Phi(T_u;U).
\]
This proves the current-minimizing property.



For a set $F$ of locally finite perimeter with
$E_u\triangle F\Subset U$, take
\[
 X=\partial\bigl(\llbracketbracket{F}
                  -\llbracketbracket{E_u}\bigr).
\]
Then $X$ is a rectifiable current with
$\partial X=0$ and $\supp X\Subset U$, and
$
 T_u+X=\partial\llbracketbracket{F}.
$
The current inequality therefore yields
\[
 P_\Phi(E_u,U)\leq P_\Phi(F,U).
\]
Thus $E_u$ is a global $\Phi$-minimizer.


 Finally, if $(x,t)\in E_u$, then
$t-s<u(x)$ for every $s>0$, and hence
$E_u-se_{n+1}\subset E_u$.  This is \eqref{eq:directed-original}.
\end{proof}

The following lemma states that a non-flat hypersurface has two points with a fixed separation of normals.

\begin{lemma}[Gauss-image gap]\label{lem:gaussgap}
Let  $\Sigma\subset\R^{n+1}$ be a
connected complete $C^2$ $\Phi$-stationary hypersurface which is
an entire graph in some direction. If $\Sigma$ is not a hyperplane,
then, for every $p\in\Sigma$, there exists $q\in\Sigma$ such that
\begin{equation}\label{eq:normal-gap}
 \nu(p)\cdot\nu(q)\le\frac12,
 \qquad |\nu(p)-\nu(q)|\ge1,
\end{equation}
where $\nu$ is the outer unit normal of the corresponding subgraph.
\end{lemma}


\begin{proof}
Fix $p\in\Sigma$ and set $N=\nu(p)$.  Assume by contradiction that
\begin{equation} \label{eq:hemisphere}
 \nu(x)\cdot N>\frac12\qquad\text{for all }x\in\Sigma.
 \end{equation}
Let $\pi_N:\Sigma\to N^\perp$ be orthogonal projection.  For every
$v\in T_x\Sigma$, write $N=N^T+(N\cdot\nu)\nu$, where $N^T$ is tangent to
$\Sigma$.  Since $d\pi_N(v)=v-(v\cdot N)N$ and
$v\cdot N=v\cdot N^T$, Cauchy-Schwarz gives
\[
 |d\pi_N(v)|^2
 =|v|^2-|v\cdot N^T|^2
 \ge \bigl(1-|N^T|^2\bigr)|v|^2
 =|\nu(x)\cdot N|^2|v|^2.
\]
Consequently,
\begin{equation}\label{eq:coplip}
 |(d\pi_N)_x(v)|
 \ge |\nu(x)\cdot N|\,|v|
 \ge\frac12|v|.
\end{equation}
Thus $(d\pi_N)_x:T_x\Sigma\to N^\perp$ is a linear isomorphism
at every $x\in\Sigma$, and
\[
 \bigl|((d\pi_N)_x)^{-1}w\bigr|\le 2|w|
 \qquad\text{for every }w\in N^\perp.
\]
In particular, $\pi_N$ is a local diffeomorphism.

Since $\Sigma$ is complete in its induced metric and
\eqref{eq:coplip} gives a uniform positive lower bound for the
differential, \cite[Theorem~1]{WolfGriffiths1963} implies that
$\pi_N:\Sigma\to N^\perp$ is a surjective covering map.
Since $N^\perp\simeq\R^n$ is simply connected and $\Sigma$ is
connected, $\pi_N$ is a global diffeomorphism.

After rotating coordinates and transforming $\Phi$ accordingly,
we may assume $N=e_{n+1}$ and write
\[
 \Sigma=\{(x,v(x)):x\in\R^n\}
 \qquad\text{for some }v\in C^2(\R^n).
\]
Since $\nu\cdot e_{n+1}>0$, its  unit normal is
$(-Dv,1)/\sqrt{1+|Dv|^2}$. Thus \eqref{eq:hemisphere} becomes
\[
 \frac1{\sqrt{1+|Dv|^2}}>\frac12.
\]
In particular,
\begin{equation}\label{eq:bounded-slope}
 |Dv|<\sqrt3.
\end{equation}
Let $G$ be the graph integrand in these rotated coordinates.
Since $\Sigma$ is $\Phi$-stationary, $v$ satisfies
\[
 \diver\bigl(DG(Dv)\bigr)=0
 \qquad\text{in }\R^n.
\]
 Parametric
ellipticity and \eqref{eq:bounded-slope} imply that $D^2G(Dv)$ is uniformly
elliptic and bounded on $\R^n$.  Up to the
fixed rotation,
\[
 G(p)=\Phi(-p,1),
 \qquad
 D^2G(p)[\zeta,\zeta]
 =D^2\Phi((-p,1))[(-\zeta,0),(-\zeta,0)].
\]
Since the projection of $(-\zeta,0)$ onto $(-p,1)^\perp$ has length at least
$|\zeta|/\sqrt{1+|p|^2}$ and
$D^2\Phi(t\xi)=t^{-1}D^2\Phi(\xi)$, the parametric
ellipticity bounds on the unit sphere therefore give constants
$0<\lambda_1\le\Lambda_1<\infty$ such that
\[
 \lambda_1|\zeta|^2\le D^2G(p)[\zeta,\zeta]
 \le\Lambda_1|\zeta|^2
 \qquad(|p|\le\sqrt3).
\]

For completeness, fix a coordinate direction $e_\alpha$ and $h\ne0$, the difference quotient
$w_{\alpha,h}(x)=(v(x+he_\alpha)-v(x))/h$ satisfies
\[
 \diver(A_{\alpha,h}(x)Dw_{\alpha,h})=0,
\]
where
\[
 A_{\alpha,h}(x)=\int_0^1D^2G\bigl((1-s)Dv(x)+sDv(x+he_\alpha)\bigr)\,ds.
\]
Since
$v\in C^2$, letting $h\to0$ in the weak formulation shows that each
$w_\alpha=\partial_\alpha v$ is a bounded entire weak solution of
\begin{equation}\label{eq:linearized-gradient}
 \diver\bigl(A(x)Dw_\alpha\bigr)=0,
 \qquad A(x):=D^2G(Dv(x)).
\end{equation}
The De Giorgi-Nash oscillation estimate
\cite{GilbargTrudinger1983}
therefore gives, for
$0<r<R/2$,
\begin{equation}\label{eq:osc-decay}
 \osc_{B_r}w_\alpha
 \le C\left(\frac rR\right)^\beta\osc_{B_R}w_\alpha,
\end{equation}
where $C$ and $\beta>0$ depend only on $n,\lambda_1$ and $\Lambda_1$.
Letting
$R\to\infty$ with $r$ fixed yields
$\osc_{B_r}w_\alpha=0$.  The arbitrariness of $r$ implies that $w_\alpha$ is constant, hence $v$ is affine and $\Sigma$ is a hyperplane, which is a contradiction with the hypothesis.  Therefore, there exists some $q$ satisfies
$\nu(p)\cdot\nu(q)\le1/2$, and the  inequality
\eqref{eq:normal-gap} follows from
\[
 |\nu(p)-\nu(q)|^2=2-2\nu(p)\cdot\nu(q)\ge1.
\]
\end{proof}


We now use a standard compactness statement to claim that the $\Phi_j$-minimizers  convergence to a nontrivial global Euclidean perimeter minimizer.


For every oriented rectifiable \(n\)-current \(S\), we write
\(S=\vec S\,\|S\|\), meaning that
\[
 S(\omega)
 =
 \int \langle\omega,\vec S\rangle\,d\|S\|
\]
for every compactly supported smooth \(n\)-form \(\omega\).
We choose representatives of sets of locally finite perimeter such that
\[
 \partial A=\operatorname{spt}|D\chi_A|.
\]

\begin{lemma}[Compactness]\label{lem:compactness}
Let $\Psi_j$ be positive convex one-homogeneous integrands such that
\[
 \Psi_j\longrightarrow |\cdot|
 \quad\text{uniformly on }\Sph^n,
 \tag{11}\label{eq:uniform-integrand}
\]
with uniform ellipticity and uniform $C^{1,1}(\Sph^n)$ bounds.  Let $E_j$ be
global $\Psi_j$-minimizers and suppose that $x_j\in\partial E_j$ belongs to a
fixed compact set.  Then, up to a subsequence, we have
\begin{equation}\label{eq:l1conv}
x_j\to x_{\infty} \quad\text{and}\quad \chi_{E_j}\to\chi_E\quad\text{in }L^1_{\mathrm{loc}}(\R^{n+1}),
\end{equation}
where $x_{\infty}\in\partial E$ and $E$ is a global Euclidean perimeter minimizer.

If $E$ coincides with a halfspace $H$ in $B_{2R}(x)$, then, up to a further subsequence,  we can choose $r\in(R,2R)$ such that
\begin{equation}\label{eq:strict-perimeter}
 \Per(E_j;B_r(x))\longrightarrow\Per(E;B_r(x)).
\end{equation}
Writing
$
 T_j=\partial\llbracketbracket{E_j}$ and
$T=\partial\llbracketbracket{E},$
we have
\[
 T_j\rightharpoonup T
 \quad\text{as currents in }B_r(x),
 \qquad
 \|T_j\|(B_r(x))\longrightarrow\|T\|(B_r(x)).
\]
In $B_r(x)$, extend $\vec T$ as the constant oriented unit tangent
$n$-vector of $\partial H$. Then, for every $\eta\in C_c(B_r(x))$,
\begin{equation}\label{eq:tiltconv}
 \int\eta\,|\vec T_j-\vec T|^2\,d\|T_j\|\longrightarrow0.
\end{equation}
\end{lemma}

\begin{proof}
Uniform convergence to the Euclidean norm and one-homogeneity imply
that there exist $\varepsilon_j\to0$ such that, for all
$\xi\in\R^{n+1}$,
\begin{equation}\label{eq:elliptic-comparison}
 (1-\varepsilon_j)|\xi|
 \le\Psi_j(\xi)\le(1+\varepsilon_j)|\xi|.
\end{equation}
We may assume that
$\varepsilon_j<1/2$, then there are constants
$0<\lambda_0\le\Lambda_0<\infty$, independent of $j$, such that
\[
 \lambda_0|\xi|\le\Psi_j(\xi)\le\Lambda_0|\xi|
 \qquad\text{for all }\xi\in\R^{n+1}.
\]


We first prove uniform density estimates. Fix
$y\in\partial E_j$ and set
\[
 V_j(s)=|E_j\cap B_s(y)|,
 \qquad W_j(s)=|B_s(y)\setminus E_j|.
\]
For almost every $s>0$, the coarea formula gives
\[
 V_j'(s)=\Haus^n(E_j^{(1)}\cap\partial B_s(y)),
 \qquad
 W_j'(s)=\Haus^n(E_j^{(0)}\cap\partial B_s(y)),
\]
and $|D\chi_{E_j}|(\partial B_s(y))=0$.


Let $F := E_j \setminus \overline{B_s(y)}$, then \(E_j \triangle F \subset B_{s+\delta}(y)\). By the minimality of \(E_j\) and the fact that $E_j=F$ on $B_{s+\delta}\setminus \overline{B_{s}}$, we obtain
\[
P_{\Psi_j}(E_j; B_s(y)) \leq P_{\Psi_j}(F; B_s(y))= \int_{E_j^{(1)} \cap \partial B_s(y)} \Psi_j(-\nu_{B_s}) \, d\mathcal{H}^n\le\Lambda_0V_j'(s).
\]
Similarly, by setting $F:=E_j\cup B_s(y)$, we get
\[
 \Per_{\Psi_j}(E_j;B_s(y))
 \le\int_{E_j^{(0)}\cap\partial B_s(y)}
        \Psi_j(\nu_{B_s})\,d\Haus^n
 \le\Lambda_0W_j'(s).
\]
Then, it follows from $\lambda_0|\xi|\leq\Psi_j(\xi)\leq\Lambda_0|\xi|$ that
\begin{equation}\label{eq:ball-comparison}
 \Per(E_j;B_s(y))
 \le\frac{\Lambda_0}{\lambda_0}V_j'(s),
 \quad\text{and}\quad
 \Per(E_j;B_s(y))
 \le\frac{\Lambda_0}{\lambda_0}W_j'(s).
\end{equation}
The intersection formula for perimeter now gives
\[
 \begin{aligned}
 \Per(E_j\cap B_s(y);\R^{n+1})
 &\le\Per(E_j;B_s(y))+V_j'(s)\\
 &\le\Bigl(1+\frac{\Lambda_0}{\lambda_0}\Bigr)V_j'(s).
 \end{aligned}
\]
By the Euclidean isoperimetric inequality,
\[
 V_j(s)^{(n)/(n+1)}\le C V_j'(s)
 \quad\text{for a.e. }s>0.
\]
Since \(y\in\operatorname{spt}|D\chi_{E_j}|\), both $V_j$ and $W_j$  are positive for every $s>0$.
Thus
\[
 \frac{d}{ds}V_j(s)^{1/(n+1)}\ge c
 \qquad\text{for almost every }s>0.
\]
Integrating on $(a,s)$ and letting $a\to0$, we obtain
\[
 V_j(s)\ge c_0s^{n+1}.
\]
Applying the same argument to \(W_j\) with the second inequality in \eqref{eq:ball-comparison}, gives
\begin{equation}\label{eq:density}
 \min\bigl\{|E_j\cap B_s(y)|,
 |B_s(y)\setminus E_j|\bigr\}\ge c_0s^{n+1}.
\end{equation}
Combining the relative isoperimetric inequality in  $B_s(y)$ with \eqref{eq:density}, we obtain the lower perimeter density
\[
 c_1s^n\le\Per(E_j;B_s(y)).
\]
Finally, by the mean value theorem for integrals, there exists
$s\in(\rho,2\rho)$ such that $V_j'(s)\le C_n\rho^n$.  Since \(B_\rho(y)\subset B_s(y)\), \eqref{eq:ball-comparison} gives
\begin{equation}\label{eq:perimeter-upper}
 \Per(E_j;B_\rho(y))\le C_0\rho^n,
\end{equation}
where the constants are independent of $j,y,$ and $\rho$.

After passing to a subsequence, we have $x_j\to x_\infty$.
Since the points $x_j$ lie in a fixed compact set,
applying \eqref{eq:perimeter-upper}   at $y=x_j$ gives uniform
perimeter bounds on every bounded set. Hence the functions
$\chi_{E_j}$ are locally uniformly bounded in $BV$.
By $BV$ compactness \cite{AFP2000} and a diagonal argument,
we obtain \eqref{eq:l1conv} and
\[
 D\chi_{E_j}\stackrel*\rightharpoonup D\chi_E
 \qquad\text{locally in }\R^{n+1}.
\]
For each $s>0$, the inclusion
$B_{s/2}(x_j)\subset B_s(x_\infty)$ holds for all sufficiently
large $j$. Passing to the limit in \eqref{eq:density} therefore
shows that both $E$ and its complement have positive measure
in every ball centered at $x_\infty$. Thus
$x_\infty\in\partial E$.


We next show that $E$ is a global Euclidean perimeter minimizer.
Let $U\Subset\R^{n+1}$ be open and let $F\triangle E\Subset U$.
The case $\Per(F;U)=\infty$ is trivial, so assume that
$\Per(F;U)<\infty$. Choose $g\in C_c^\infty(U)$, $0\le g\le1$,
such that $g=1$ on a neighborhood of
$\supp(\chi_F-\chi_E)$, and set
\[
 V_t=\{g>t\},\qquad 0<t<1.
\]
Then $F=E$ in a neighborhood of $\partial V_t$.
 For almost every $t$, the value $t$ is regular and
\[
 |D\chi_{E_j}|(\partial V_t)
 =
 |D\chi_E|(\partial V_t)
 =
 |D\chi_F|(\partial V_t)
 =0
 \qquad\text{for every }j.
\]
At these levels, the interior and exterior traces agree;
we denote their common value by $\Tr$.
Passing to a further subsequence, we may assume that
\[
 \sum_j
 \|\chi_{E_j}-\chi_E\|_{L^1(\{0<g<1\})}<\infty.
\]
The coarea formula and the trace formula give
\[
 \int_0^1\int_{\partial V_t}
 |\Tr\chi_{E_j}-\Tr\chi_E|\,d\Haus^n\,dt
 \le
 \|\nabla g\|_\infty
 \|\chi_{E_j}-\chi_E\|_{L^1(\{0<g<1\})}.
\]
Summing in $j$ and applying Fubini's theorem, we may choose
one such regular value $t$, independent of $j$, so that,
with $V=V_t$,
\begin{equation}\label{eq:trace-convergence}
 \int_{\partial V}
 |\Tr\chi_{E_j}-\Tr\chi_E|\,d\Haus^n
 \longrightarrow0.
\end{equation}

\medskip
Define
\[
 F_j=(F\cap V)\cup(E_j\setminus\overline V),
\]
and choose an open set $W$ with
$\overline V\Subset W\Subset U$.
By minimality of $E_j$ in $W$, the gluing formula
\cite{AFP2000}, and the bound $\Psi_j\le\Lambda_0|\cdot|$,
\begin{equation}\label{eq:gluing-limit}
 P_{\Psi_j}(E_j;V)
 \le
 P_{\Psi_j}(F;V)
 +\Lambda_0\int_{\partial V}
 |\operatorname{Tr}\chi_{E_j}-\operatorname{Tr}\chi_F|\,d\Haus^n.
\end{equation}
Here the common perimeter in $W\setminus\overline V$
has been cancelled.

Since \(F=E\) near \(\partial V\), the trace term vanishes as \(j\to\infty\) by \eqref{eq:trace-convergence}. Uniform convergence of \(\Psi_j\) yields \(P_{\Psi_j}(F;V)\to P(F;V)\), while lower semicontinuity and \eqref{eq:elliptic-comparison} give
\[
 P(E;V)
 \le\liminf_j P(E_j;V)
 \le\liminf_j\frac{P_{\Psi_j}(E_j;V)}{1-\varepsilon_j}
 \le\limsup_j\frac{P_{\Psi_j}(E_j;V)}{1-\varepsilon_j}
 \le P(F;V),
\]
where we used \eqref{eq:gluing-limit} in the last inequality. Since \(E\) and \(F\) coincide outside \(V\) and in a neighborhood of \(\partial V\), we obtain
\(P(E;U)\le P(F;U)\). Thus \(E\) is a global Euclidean perimeter minimizer.

\medskip
It remains to prove the convergence statements when $E$ coincides
with a halfspace in $B_{2R}(x)$.
Define
\[
 Q_j:=\llbracketbracket{E_j}-\llbracketbracket{E}\quad\text{and}\quad d_x(z)=|z-x|.
\]
By local $L^1$ convergence,
\[
 \Mass(Q_j\llcorner B_{2R}(x))
 =\int_{B_{2R}(x)}|\chi_{E_j}-\chi_E|\,dz\longrightarrow0.
\]
Passing to a subsequence, we may assume \(\sum_j \mathbf{M}(Q_j\llcorner B_{2R}(x))<\infty\), the slicing inequality
\cite{Maggi2012} yields
\[
 \int_R^{2R}\Mass(\langle Q_j,d_x,r\rangle)\,dr
 \le\Mass(Q_j\llcorner(B_{2R}(x)\setminus B_R(x))),
\]
where  \(\langle Q_j,d_x,r\rangle\) is the slice of \(Q_j\) on \(\{d_x=r\}=\partial B_r(x)\).
By the slicing estimate and Fubini's theorem, we can choose $r\in(R,2R)$ such that
\[
 |D\chi_{E_j}|(\partial B_r(x))
 =
 |D\chi_E|(\partial B_r(x))
 =0
 \qquad\text{for every }j,
\]
and
\begin{equation}\label{eq:slice-small}
 \Mass(\langle Q_j,d_x,r\rangle)\longrightarrow0.
\end{equation}


Define
\[
 F_j=(E\cap B_r(x))\cup(E_j\setminus\overline{B_r(x)}).
\]
Then
$
 \llbracketbracket{F_j}
 =
 \llbracketbracket{E_j}-Q_j\llcorner B_r(x).
$
Using the slicing identity
$
 \langle Q_j,d_x,r\rangle
 =
 \partial(Q_j\llcorner B_r(x))
 -
 (\partial Q_j)\llcorner B_r(x),
$
we obtain
\[
 \begin{aligned}
 \partial\llbracketbracket{F_j}
 &=
 T_j-\bigl((T_j-T)\llcorner B_r(x)
          +\langle Q_j,d_x,r\rangle\bigr)\\
 &=
 T\llcorner B_r(x)
 +T_j\llcorner
   \bigl(\R^{n+1}\setminus\overline{B_r(x)}\bigr)
 -\langle Q_j,d_x,r\rangle.
 \end{aligned}
\]
Choose $s$ with $r<s<2R$. By minimality of \(E_j\) in \(B_s(x)\), and the fact that \(E_j=F_j\) on \(B_s\setminus\overline{B_r}\),  we find
\begin{equation}\label{eq:strict-gluing}
 \Per_{\Psi_j}(E_j;B_r(x))
 \le
 \Per_{\Psi_j}(E;B_r(x))
 +\Lambda_0\Mass(\langle Q_j,d_x,r\rangle).
\end{equation}

By \eqref{eq:elliptic-comparison},
$
 \Per_{\Psi_j}(E;B_r(x))
 \longrightarrow\Per(E;B_r(x)).
$
Combining this with \eqref{eq:slice-small} and
\eqref{eq:strict-gluing} gives
\[
 \limsup_j\Per_{\Psi_j}(E_j;B_r(x))
 \le\Per(E;B_r(x)).
\]
Using again \eqref{eq:elliptic-comparison}
together with lower semicontinuity, yields
\[
 \Per(E;B_r(x))
 \le\liminf_j\Per(E_j;B_r(x))
 \le\limsup_j\Per(E_j;B_r(x))
 \le\Per(E;B_r(x)).
\]
This proves \eqref{eq:strict-perimeter}.
Together with the $L^1$ convergence, we therefore have
\[
 D\chi_{E_j}\stackrel*\rightharpoonup D\chi_E
 \quad\text{in }B_r(x),
 \qquad
 |D\chi_{E_j}|(B_r(x))
 \longrightarrow|D\chi_E|(B_r(x)).
\]
Equivalently, for the associated boundary currents,
\[
 T_j\rightharpoonup T
 \quad\text{in }B_r(x),
 \qquad
 \|T_j\|(B_r(x))\longrightarrow\|T\|(B_r(x)).
\]
By Reshetnyak's continuity theorem \cite{AFP2000},
\begin{equation}\label{eq:weighted-mass}
 \int\eta\,d\|T_j\|\longrightarrow\int\eta\,d\|T\|
\end{equation}
for every $\eta\in C_c(B_r(x))$.

Since $E$ coincides with the halfspace $H$ in $B_{2R}(x)$,
we extend $\vec T$ to $B_r(x)$ as the constant oriented unit
tangent $n$-vector of $\partial H$.
Let $\omega_{\vec T}$ be its dual constant $n$-covector.
For every nonnegative
$\eta\in C_c(B_r(x))$,
\[
 \begin{aligned}
 \int\eta|\vec T_j-\vec T|^2\,d\|T_j\|
 &=
 2\int\eta\,d\|T_j\|
 -2T_j(\eta\omega_{\vec T})\\
 &\longrightarrow
 2\int\eta\,d\|T\|
 -2T(\eta\omega_{\vec T})
 =0.
 \end{aligned}
\]
For signed $\eta$, the same conclusion follows from
\[
 \left|
 \int\eta|\vec T_j-\vec T|^2\,d\|T_j\|
 \right|
 \le
 \int|\eta||\vec T_j-\vec T|^2\,d\|T_j\|
 \longrightarrow0.
\]
This proves \eqref{eq:tiltconv}.

\end{proof}

\begin{lemma}\label{lem:directedclosed}
Assume $\chi_{E_j}\to\chi_E$ in $L^1_{\mathrm{loc}}(\R^{n+1})$,
$a_j\to a\in\Sph^n$, and, for every $t>0$,
\begin{equation}\label{eq:direction-j}
 \chi_{E_j}(x+ta_j)\ge\chi_{E_j}(x)
 \quad\text{for a.e. }x\in\R^{n+1}.
\end{equation}
Then, for every $t>0$,
\[
 \chi_E(x+ta)\ge\chi_E(x)
 \quad\text{for a.e. }x\in\R^{n+1}.
\]
\end{lemma}

 \begin{proof}
Fix $t>0$ and a compact set $K\subset\R^{n+1}$.
Since $(a_j)$ is bounded, we can choose a compact set $K'$
containing $K$, $K+ta$, and $K+ta_j$ for every $j$.
The triangle inequality gives
\[
 \begin{aligned}
 &\|\chi_{E_j}(\,\cdot+ta_j)
       -\chi_E(\,\cdot+ta)\|_{L^1(K)}\\
 &\qquad\le
 \|\chi_{E_j}-\chi_E\|_{L^1(K')}
 +
 \|\chi_E(\,\cdot+ta_j)
       -\chi_E(\,\cdot+ta)\|_{L^1(K)}.
 \end{aligned}
\]
 By local $L^1$ convergence and
continuity of translations in $L^1_{loc}(\mathbb{R}^{n+1})$, both terms tend to zero,
Consequently,
\[
 \chi_{E_j}(\,\cdot+ta_j)-\chi_{E_j}
 \longrightarrow
 \chi_E(\,\cdot+ta)-\chi_E
 \quad\text{in }L^1(K).
\]
Since $K$ and $t>0$ were arbitrary, the conclusion follows by passing to the limit as \(j\to\infty\)
\end{proof}


\section{Classification of Directed Euclidean Minimizers}\label{sec:directed-rigidity}

In this section, we prove that, for $n\le7$, every nontrivial
directed Euclidean perimeter minimizer is a halfspace.

\begin{proposition}
\label{prop:directed-classification}
Let $1\le n\le7$. Suppose that $E\subset\R^{n+1}$ is a nontrivial
global Euclidean perimeter minimizer and, for some $a\in\Sph^n$,
\begin{equation}\label{eq:directed-E}
 \chi_E(x+ta)\ge\chi_E(x)
 \quad\text{for every }t>0\text{ and a.e. }x.
\end{equation}
Then $E$ is a halfspace. The direction $a$ is allowed to be
tangent to its boundary.
\end{proposition}

\begin{proof}
We first recall the classification of directed minimizing cones.
Let $C$ be a nontrivial perimeter-minimizing cone with vertex
at the origin satisfying \eqref{eq:directed-E}.

For $n\le6$, the regularity theorem
\cite[Theorem~28.1]{Maggi2012} implies that the vertex is regular.
Since $C$ is a cone, it coincides with its tangent halfspace
at the origin.



For $n=7$, we apply \cite[Corollary~4.6]{Schnuerer2003} to
the open representative of $C$. Integrating
\eqref{eq:directed-E} over balls shows that $C$ is directed
with respect to $-a$ in the convention of
\cite[Definition~3.1]{Schnuerer2003}.
That corollary therefore implies that $C$ is a halfspace.

We now return to $E$. Since $E$ is nontrivial,
$|D\chi_E|\not\equiv0$, and hence $\partial^*E$ is nonempty.
Choose $x_0\in\partial^*E$ and define
\[
 \Theta_E(r)
 =
 \frac{\Per(E;B_r(x_0))}{\omega_n r^n},
 \qquad r>0,
\]
where $\omega_n$ is the volume of the unit ball in $\R^n$.
The boundary varifold of a perimeter minimizer is stationary.
Thus the monotonicity formula \cite{Simon1983} gives, for
$0<\rho<R$,
\begin{equation}\label{eq:monotonicity-exact}
 \Theta_E(R)-\Theta_E(\rho)
 =
 \frac1{\omega_n}
 \int_{B_R(x_0)\setminus B_\rho(x_0)}
 \frac{|(x-x_0)^\perp|^2}{|x-x_0|^{n+2}}
 \,d|D\chi_E|(x)
 \ge0,
\end{equation}
where $(x-x_0)^\perp$ denotes the projection onto the normal
space of $\partial^*E$ at $x$.
De Giorgi's blow-up theorem at $x_0$ gives
\begin{equation}\label{eq:density-zero}
 \Theta_E(0+)=1.
\end{equation}
Applying the perimeter bound \eqref{eq:perimeter-upper}  to
the Euclidean integrand, we have
\[
 \Per(E;B_R(x_0))\le C_nR^n
 \qquad\text{for every }R>0.
\]
Consequently, $\Theta_E$ is bounded and nondecreasing, so
\[
 \Theta_\infty
 :=
 \lim_{R\to\infty}\Theta_E(R)
 <\infty.
\]
Let $R_j\to\infty$ and consider the blow-downs
\[
 E_j:=\frac{E-x_0}{R_j}.
\]
Each $E_j$ is a global Euclidean perimeter minimizer  directed along $a$, and $0\in\partial E_j$.
Applying Lemma~\ref{lem:compactness} with
$\Psi_j=|\cdot|$, we obtain, after passing to a subsequence,
\begin{equation}\label{eq:blowdown-l1}
 \chi_{E_j}\longrightarrow\chi_C
 \quad\text{in }L^1_{\mathrm{loc}}(\R^{n+1}),
\end{equation}
where $C$ is a global Euclidean perimeter minimizer and
$0\in\partial C$. In particular, $C$ is nontrivial.
By Lemma~\ref{lem:directedclosed}, $C$ is directed along $a$.

We next establish convergence of the perimeters.
By the coarea formula and a diagonal subsequence,
\[
 \int_{\partial B_r}
 |\Tr\chi_{E_j}-\Tr\chi_C|\,d\Haus^n
 \longrightarrow0
 \qquad\text{for almost every }r>0,
\]
where balls are centered at the origin.
For almost every such $r$, we also have
\[
 |D\chi_{E_j}|(\partial B_r)
 =
 |D\chi_C|(\partial B_r)
 =0
 \qquad\text{for every }j.
\]
For these $r$, apply the gluing estimate used in
\eqref{eq:gluing-limit} to
$
 F_j=(C\cap B_r)\cup(E_j\setminus\overline{B_r}),
$
we obtain
\[
 \Per(E_j;B_r)
 \le
 \Per(C;B_r)
 +
 \int_{\partial B_r}
 |\Tr\chi_{E_j}-\Tr\chi_C|\,d\Haus^n.
\]
Together with lower semicontinuity, this proves
\begin{equation}\label{eq:blowdown-strict}
 \Per(E_j;B_r)\longrightarrow\Per(C;B_r)
 \quad\text{for almost every }r>0.
\end{equation}

Scaling and \eqref{eq:blowdown-strict} now yield, for almost
every $r>0$,
\begin{equation}\label{eq:blowdown-density}
 \Theta_C(r)
 :=
 \frac{\Per(C;B_r)}{\omega_n r^n}
 =
 \lim_{j\to\infty}
 \frac{\Per(E;B_{rR_j}(x_0))}
      {\omega_n(rR_j)^n}
 =
 \Theta_\infty.
\end{equation}
Since $\Theta_C$ is nondecreasing, this equality holds for
every $r>0$. The equality case in the monotonicity formula
implies that $C$ is a cone with vertex at the origin.
The directed cone classification established above therefore
shows that $C$ is a halfspace through the origin.
Hence
$
 \Per(C;B_r)=\omega_n r^n,
$
and \eqref{eq:blowdown-density} gives $\Theta_\infty=1$.

Finally, \eqref{eq:density-zero} and monotonicity imply
\[
 1=\Theta_E(0+)\le\Theta_E(r)\le\Theta_\infty=1
 \qquad\text{for every }r>0.
\]
Thus equality holds in \eqref{eq:monotonicity-exact}, and
$(x-x_0)^\perp=0$
 for $|D\chi_E|$-almost every $x.$
By the equality case of the monotonicity formula,
$E-x_0$ is a cone. Consequently,
\[
 E_j=\frac{E-x_0}{R_j}=E-x_0
\]
up to sets of measure zero.
Combining this with \eqref{eq:blowdown-l1} gives $E-x_0=C$.
Since $C$ is a halfspace, we conclude  that $E$ is a halfspace.


\end{proof}

\section{Flat Compactness and Proof  Theorem \ref{thm:main}}
\label{sec:flatness-main-proof}

In this section, we prove Theorem~\ref{thm:main}.
We first recall Figalli's regularity theorem
\cite[Theorem~2.1]{Figalli2017}, retaining the numbering of
the structural assumptions in that paper.


Let $\Lambda_n(\R^{n+1})$ be the space of $n$-vectors in
$\R^{n+1}$, and let
$\mathcal F:\Lambda_n(\R^{n+1})\to\R$ be of class $C^1$
outside the origin. Assume that there exist constants
$A_0,\rho_0>0$ such that
\begin{align}
&\mathcal F(\lambda \xi)=\lambda \mathcal F(\xi)
\qquad
\forall \lambda\ge0,\ \xi\in\Lambda_n(\R^{n+1}),
\tag{2.2}\\
&\mathcal F(\xi)\ge1
\qquad
\forall |\xi|=1,
\tag{2.3}\\
&\mathcal F(\eta)-\langle d\mathcal F(\xi),\eta\rangle\ge0
\qquad
\forall |\xi|=|\eta|=1,
\tag{2.4}\\
&\mathcal F(\eta)-\langle d\mathcal F(\xi),\eta\rangle
\ge\frac12|\xi-\eta|^2
\qquad
\forall |\xi|=|\eta|=1,\ |\xi-\pmb e_0|\le\rho_0,
\tag{2.5}\\
&\sup_{|\xi|=1}
\bigl(\mathcal F(\xi)+|d\mathcal F(\xi)|\bigr)\le A_0,
\tag{2.6}\\
&\mathcal F(\eta)-\langle d\mathcal F(\xi),\eta\rangle
\le A_0|\xi-\eta|^2
\qquad
\forall |\xi|=|\eta|=1,\ |\xi-\pmb e_0|\le\rho_0,
\tag{2.7}
\end{align}
where
$\pmb e_0:=e_1\wedge\cdots\wedge e_n\in\Lambda_n(\R^{n+1})$
and $d\mathcal F$ denotes the differential of $\mathcal F$.

For a rectifiable $n$-current $S$ of finite mass, set
\[
\mathbb F(S):=\int\mathcal F(\vec S)\,d\|S\|.
\]
Denote
\[
B_r:=B_r(0)\subset\R^n,
\qquad
\mathcal C_r:=B_r\times\R,
\qquad
p(x,y):=x.
\]
For $r>0$ and $u\in C^1(B_r)$, define the cylindrical excess
\[
\mathcal E(T,r,u)
:=
\frac1{r^n}
\int_{\mathcal C_r}|\vec T-\vec U|^2\,d\|T\|,
\]
where
\[
U(x):=(x,u(x)),
\qquad
\Lambda_n\nabla U(x)
:=
\partial_1U(x)\wedge\cdots\wedge\partial_nU(x),
\]
and
\[
\vec U(z)
:=
\frac{\Lambda_n\nabla U(p(z))}
     {|\Lambda_n\nabla U(p(z))|}.
\]

 \begin{lemma}[{\cite[Theorem~2.1]{Figalli2017}}]
\label{lemma4.1}
\label{thm:figalli-2017}
Let $\mathcal F:\Lambda_n(\R^{n+1})\to\R$ be of class $C^1$
outside the origin and satisfy $(2.2)$--$(2.7)$.
Let $T$ be an $n$-dimensional integer rectifiable current
of finite mass in $\R^{n+1}$ satisfying:
\begin{itemize}
\item[(H1)]
$0\in\supp(T)$, $\supp(T)\subset\overline{\mathcal C_R}$
is compact, and
$\supp(\partial T)\subset\partial\mathcal C_R$;

\item[(H2)]
$T$ is $\mathcal F$-minimal in $\mathcal C_R$, that is,
\[
\mathbb F(T)\le\mathbb F(T+X)
\]
for every rectifiable $n$-current $X$ of finite mass with
$\partial X=0$ and $\supp X\subset\mathcal C_R$;

\item[(H3)]
\[
p_\#(T\llcorner\mathcal C_R)
=
\llbracketbracket{B_R},
\]
where $\llbracketbracket{B_R}$ is the current of integration
over $B_R$ with its standard orientation.
\end{itemize}

There exist constants $\epsilon_0>0$ and $\delta\in(0,1)$,
depending only on $n,A_0,\rho_0$, such that, if
\[
\mathcal E(T,R,0)\le\epsilon_0,
\]
then there exists $u\in C^{1,\delta}(B_{R/2})$ such that
\[
T\llcorner\mathcal C_{R/2}
=
U_\#\llbracketbracket{B_{R/2}},
\qquad
U(x):=(x,u(x)).
\]
In particular,
$
\mathcal E(T,R/2,u)=0.
$
\end{lemma}

We now apply Figalli's regularity theorem to obtain the following
flat compactness result near a multiplicity-one plane.

 \begin{lemma}\label{lem:c1compactness}
Let $\Psi_j$ be convex one-homogeneous integrands with uniform
parametric ellipticity and uniform $C^{1,1}(\Sph^n)$ bounds, and assume
\[
 \Psi_j\to|\cdot|\quad\text{in }C^1(\Sph^n).
\]
Let $E_j$ be sets of locally finite perimeter in $B_4$, and let
$T_j:=\partial\llbracketbracket{E_j}$ be their oriented boundary
currents in $B_4$. Assume that there exists a $C^1$ closed
$n$-form $\omega_j$ on $\R^{n+1}$ such that
\[
 \begin{gathered}
  \omega_j(\xi)\le\Psi_j(\nu_\xi)
  \quad\text{pointwise in $\R^{n+1}$ for every oriented unit
  $n$-vector $\xi$},\\
  \omega_j(\vec T_j)=\Psi_j(\nu_{\vec T_j})
  \quad\|T_j\|\text{-a.e. in }B_4.
 \end{gathered}
\]
Suppose that, for a multiplicity-one oriented hyperplane $P$
through the origin,
\begin{equation}\label{eq:strict-plane}
 T_j\to T_P\quad\text{as currents},
 \qquad
 \|T_j\|\stackrel*\rightharpoonup\|T_P\|
 \quad\text{as Radon measures locally in }B_4.
\end{equation}
Then the supports \(\operatorname{spt}T_j\) are eventually locally single \(C^{1,\alpha}\) graphs over \(P\) for some
$\alpha\in(0,1)$ independent of $j$. Their $C^{1,\alpha}$ bounds
are uniform in $j$ on compact subsets, and their oriented unit
normals converge uniformly to $\nu_P$ on compact subsets of $B_4$.
\end{lemma}

\begin{proof}
After a fixed rotation, suppose that
\[
 P=\R^n\times\{0\},
 \qquad
 \vec T_P=e_0:=e_1\wedge\cdots\wedge e_n,
\]
and set $\Cyl_\rho:=B_\rho^n\times\R, $ $ p(x,t):=x.$
Without loss of generality, we may assume that
\[
 \frac12\le\Psi_j(\nu)\le2
 \qquad\text{for every }\nu\in\Sph^n
\]
for $j$ sufficient large.

We first verify the structural assumptions in
Lemma~\ref{lemma4.1}. Define
\[
 F_j(0):=0,
 \qquad
 F_j(\xi):=
 |\xi|\Psi_j(\nu_{\xi/|\xi|})
 \quad\text{for }\xi\ne0.
\]
The map $\xi\mapsto\nu_\xi$ on unit $n$-vectors is the restriction
of a fixed linear isometry
\[
 \mathcal I:\Lambda_n(\R^{n+1})\longrightarrow\R^{n+1}.
\]
Thus $F_j=\Psi_j\circ\mathcal I$. In particular, $F_j$ is convex,
one-homogeneous, and of class $C^{1,1}$ locally outside the origin.

For unit $n$-vectors $\xi,\eta$, set
\[
 \mathcal B_j(\xi,\eta)
 :=F_j(\eta)-dF_j(\xi)[\eta].
\]
Convexity and one-homogeneity give
$\mathcal B_j(\xi,\eta)\ge0$.
When $\eta$ is sufficiently close to $\xi$, write
\[
 \eta=\frac{\xi+v}{|\xi+v|},
 \qquad v\perp\xi,
\]
and set $G_{j,\xi}(v):=F_j(\xi+v)$.
Euler's identity gives
\[
 \mathcal B_j(\xi,\eta)
 =
 \frac{
 G_{j,\xi}(v)-G_{j,\xi}(0)-DG_{j,\xi}(0)\cdot v
 }{\sqrt{1+|v|^2}}.
\]
Uniform parametric ellipticity and the common $C^{1,1}$ bound imply
\[
 c\Id\le D^2G_{j,\xi}(v)\le C\Id
\]
for almost every $v$ in a fixed neighborhood of $0$ in $\xi^\perp$,
with constants independent of $j$ and $\xi$.
The Taylor inequalities for $C^{1,1}$ functions, together with
$|v|\simeq|\xi-\eta|$, therefore give, for some fixed $r_0>0$,
\begin{equation}\label{eq:bregman-near}
 c_1|\xi-\eta|^2
 \le\mathcal B_j(\xi,\eta)
 \le C_1|\xi-\eta|^2
 \quad\text{if }|\xi-\eta|\le r_0.
\end{equation}

For $|\xi-\eta|\ge r_0$, use the $C^1$ convergence of $F_j$
to the Euclidean norm on the unit sphere.
For the Euclidean integrand,
\[
 |\eta|-d|\cdot|(\xi)[\eta]
 =1-\xi\cdot\eta
 =\frac12|\xi-\eta|^2.
\]
Consequently,
\[
 \sup_{|\xi|=|\eta|=1}
 \left|
 \mathcal B_j(\xi,\eta)-\frac12|\xi-\eta|^2
 \right|
 \longrightarrow0.
\]
For all sufficiently large $j$, this supremum is at most
$r_0^2/4$, and hence
\begin{equation}\label{e4.1}
 \frac14|\xi-\eta|^2
 \le\mathcal B_j(\xi,\eta)
 \le\frac34|\xi-\eta|^2
 \quad\text{if }|\xi-\eta|\ge r_0.
\end{equation}
Combining \eqref{eq:bregman-near} and \eqref{e4.1}, we obtain
constants $c_0,C_0>0$, independent of $j$, such that
\begin{equation}\label{eq:bregman-global}
 c_0|\xi-\eta|^2
 \le\mathcal B_j(\xi,\eta)
 \le C_0|\xi-\eta|^2
 \qquad
 \text{for all }|\xi|=|\eta|=1
\end{equation}
for all sufficiently large $j$.


Choose one fixed $b>0$ such that
\[
 bc_0\ge\frac12,
 \qquad b\ge2,
\]
and set $\widehat F_j:=bF_j$.
Then $\widehat F_j$ satisfies conditions $(2.2)$-$(2.7)$
with a common $A_0$ and, for example, $\rho_0=1$.
Multiplication by $b$ does not change the minimizing currents.
For a rectifiable current $S$, write
\[
 \mathbb F_j(S):=
 \int\widehat F_j(\vec S)\,d\|S\|.
\]

 We next establish the local minimality needed for the density
estimate.
Applying the calibration inequality to both $\xi$ and $-\xi$ gives
\[
 |\omega_j(z)(\xi)|
 \le
 \max\{\Psi_j(\nu_\xi),\Psi_j(-\nu_\xi)\}
 \le C
\]
for every $z\in\R^{n+1}$ and every unit $n$-vector $\xi$.
Thus the forms $\omega_j$ are uniformly bounded on $\R^{n+1}$.
Since they are closed, the cutoff argument in the proof of
Lemma~\ref{lem:calibration} gives
\[
 X(\omega_j)=0
\]
for every rectifiable $n$-current $X$ of finite mass with
$\partial X=0$.

Combining this with the calibration inequality \(\omega_j(\xi)\le \Psi_j(\nu_\xi)\)   show that $T_j$ is locally
minimizing in $B_4$ and that $E_j$ locally minimizes
$\Psi_j$-perimeter there.
The density estimate proved in Lemma~\ref{lem:compactness}
applies to balls compactly contained in $B_4$.
In particular,
\[
 \|T_j\|(B_r(y))\ge c r^n
\]
whenever $y\in\operatorname{spt}T_j$ and
$B_{2r}(y)\Subset B_4$, with $c>0$ independent of $j$.
Together with the measure convergence in \eqref{eq:strict-plane},
this implies
\[
 \operatorname{spt}T_j\longrightarrow P
 \quad\text{locally in Hausdorff distance in }B_4.
\]


 We now verify (H1)-(H3).
Fix $z\in P\cap B_1$ and choose
$z_j\in\operatorname{spt}T_j$ with $z_j\to z$.
Translate the currents $T_j$ and calibration forms   \(\omega_j\) by   $-z_j$ and
 maintain the same notation.
Then
$
 0\in\operatorname{spt}T_j,
$
and the translated currents still satisfy
\eqref{eq:strict-plane} locally in $B_2$, with the same limiting
plane $P$. 
 Set $r(x,t):=|x|$.
By the slicing theorem and the countability of the family
$\{T_j\}$, we may choose some
$
 \rho\in\left(\frac14,\frac12\right)
$
such that the slices $\langle T_j,r,\rho\rangle$ are locally
integral currents and
\[
 \|T_j\|
 \bigl((\partial B_\rho^n\times\R)\cap B_1\bigr)=0
 \qquad\text{for every }j.
\]
The corresponding zero-mass condition also holds for $T_P$, i.e. $\|T_P\|((\partial B_\rho^n\times\mathbb R)\cap B_1)=0$.

Fix $h>0$ such that
\[
 \overline{B_{1/2}^n}\times[-2h,2h]\Subset B_1.
\]
By local Hausdorff convergence, for all sufficiently large $j$,
\[
 \operatorname{spt}T_j
 \cap
 \bigl(\overline{B_\rho^n}\times[-2h,2h]\bigr)
 \subset
 \overline{B_\rho^n}\times
 \left(-\frac h2,\frac h2\right).
\]
Define
\[
 Z_{\rho,h}:=B_\rho^n\times(-h,h),
 \qquad
 S_j:=T_j\llcorner Z_{\rho,h}.
\]
 Since $\partial T_j=0$ in $B_1$, the restriction--slicing
identity gives
\[
 \operatorname{spt}(\partial S_j)
 \subset
 \partial B_\rho^n\times(-h,h)
 \subset\partial\Cyl_\rho.
\]
There are no boundary terms at $t=\pm h$, because the support
of $T_j$ is disjoint from a neighborhood of those horizontal
faces.
Moreover,
\[
 0\in\operatorname{spt}S_j,
 \qquad
 \operatorname{spt}S_j
 \subset\overline{Z_{\rho,h}}
 \subset\overline{\Cyl_\rho},
\]
and $\overline{Z_{\rho,h}}$ is compact.
The currents $S_j$ are integral currents of finite mass.
Thus (H1) holds.

Since $\|T_P\|(\partial Z_{\rho,h})=0$,
\eqref{eq:strict-plane} also gives
\[
 S_j\longrightarrow T_P\llcorner Z_{\rho,h}
 \quad\text{as currents},
 \qquad
 \mathbf M(S_j)\longrightarrow |B_\rho^n|.
\]
All these currents have support in the same compact set
$\overline{Z_{\rho,h}}$.

To verify (H2), let $X$ be a rectifiable $n$-current with
$\partial X=0$ and $\operatorname{spt}X\subset\Cyl_\rho$.
If $\mathbb F_j(S_j+X)=\infty$, the required inequality is
immediate.
Otherwise, since $\widehat F_j\ge1$ on unit $n$-vectors,
\[
 \mathbf M(X)
 \le\mathbf M(S_j+X)+\mathbf M(S_j)<\infty.
\]
it follows from Lemma~\ref{lem:calibration} that 
$X(b\omega_j)=0$.
Consequently,
\[
 \mathbb F_j(S_j+X)
 \ge (S_j+X)(b\omega_j)
 =S_j(b\omega_j)
 =\mathbb F_j(S_j).
\]
Thus (H2) holds.




Since
$\operatorname{spt}(\partial S_j)\subset\partial\Cyl_\rho$,
we have $\partial(p_\#S_j)=0$ in $B_\rho^n$.
The constancy theorem therefore gives
\begin{equation}\label{eq:degree}
 p_\#S_j=k_j\llbracketbracket{B_\rho^n}
\end{equation}
for an integer $k_j$.
The convergence of $S_j$, together with their common compact
support, implies
\[
 p_\#S_j\longrightarrow\llbracketbracket{B_\rho^n}.
\]
Hence $k_j\to1$, and integrality gives $k_j=1$ for all
sufficiently large $j$.
This proves (H3).


\smallskip

We therefore have
\[
 \mathbf M(S_j)\longrightarrow |B_\rho^n|,
 \qquad
 S_j(dx^1\wedge\cdots\wedge dx^n)=|B_\rho^n|.
\]
It follows that
\begin{equation}\label{e4.2}
 \begin{aligned}
 \mathcal E(S_j,\rho,0)
 &:=
 \rho^{-n}
 \int_{\Cyl_\rho}|\vec S_j-e_0|^2\,d\|S_j\|\\
 &=2\rho^{-n}
 \bigl(\mathbf M(S_j)-|B_\rho^n|\bigr)
 \longrightarrow0.
 \end{aligned}
\end{equation}


By Lemma~\ref{lemma4.1} and the quantitative estimates in
\cite[Section~4]{Figalli2017}, there exist
$\alpha\in(0,1)$ and $C<\infty$, independent of $j$, such that
for all sufficiently large $j$,
\[
 S_j\llcorner\Cyl_{\rho/2}
 =
 (U_j)_\#\llbracketbracket{B_{\rho/2}^n},
 \qquad
 U_j(x):=(x,f_j(x)),
\]
where
$
 f_j\in C^{1,\alpha}(B_{\rho/2}^n)$ and $
 \|f_j\|_{C^{1,\alpha}(B_{\rho/2}^n)}\le C.$
In particular,
\begin{equation}\label{eq:uniform-holder-gradient}
 [Df_j]_{C^\alpha(B_{\rho/3}^n)}\le C,
 \qquad
 \|Df_j\|_{L^\infty(B_{\rho/3}^n)}\le C.
\end{equation}
The constants are uniform because the structural constants of
$\widehat F_j$ are uniform and
$\rho\in(1/4,1/2)$.
On this graphical region,
\[
 \vec S_j(U_j(x))
 =
 \frac{
 \partial_1U_j(x)\wedge\cdots\wedge\partial_nU_j(x)
 }{\sqrt{1+|Df_j(x)|^2}}.
\]
Thus
\[
 \langle\vec S_j(U_j(x)),e_0\rangle
 =
 \frac1{\sqrt{1+|Df_j(x)|^2}},
\]
and the area formula gives
\[
 \int_{\Cyl_{\rho/3}}
 |\vec S_j-e_0|^2\,d\|S_j\|
 =
 2\int_{B_{\rho/3}^n}
 \bigl(\sqrt{1+|Df_j|^2}-1\bigr)\,dx.
\]
Since $|Df_j|\le C$,
\[
 \sqrt{1+|Df_j|^2}-1
 =
 \frac{|Df_j|^2}{\sqrt{1+|Df_j|^2}+1}
 \ge c|Df_j|^2.
\]
Together with \eqref{e4.2}, this yields
\begin{equation}\label{e4.3}
 \|Df_j\|_{L^2(B_{\rho/3}^n)}\longrightarrow0.
\end{equation}

Let $K\Subset B_{\rho/3}^n$ be compact, and set
\[
 M_j:=\max_K|Df_j|,
 \qquad
 d:=\operatorname{dist}(K,\partial B_{\rho/3}^n)>0.
\]
Enlarge the constant $C$ in
\eqref{eq:uniform-holder-gradient}, if necessary, so that $C\ge1$.
If $M_j>0$, choose $x_j\in K$ such that
$|Df_j(x_j)|=M_j$, and set
\[
 r_j:=
 \min\left\{
 \frac d2,
 \left(\frac{M_j}{2C}\right)^{1/\alpha}
 \right\}.
\]
Then $B_{r_j}(x_j)\subset B_{\rho/3}^n$, and the uniform
H\"{o}lder estimate gives
\[
 |Df_j|\ge\frac{M_j}{2}
 \quad\text{on }B_{r_j}(x_j).
\]
Since $M_j\le C$, it follows that
\[
 \int_{B_{\rho/3}^n}|Df_j|^2\,dx
 \ge cM_j^2r_j^n
 \ge c_K M_j^{2+n/\alpha},
\]
where $c_K>0$ is independent of $j$.
By \eqref{e4.3}, $M_j\to0$.
Hence
\[
 Df_j\longrightarrow0
 \quad\text{uniformly on compact subsets of }B_{\rho/3}^n.
\]
The local Hausdorff convergence also gives
$f_j\to0$ uniformly on these compact subsets.

The formula for $\vec S_j$ now implies that
$\vec S_j\to e_0$ uniformly on compact graphical regions.
Since $S_j$ agrees with $T_j$ in $Z_{\rho,h}$, the same conclusion
holds there for $\vec T_j$.
By the orientation convention and the isometry $\mathcal I$,
\[
 \nu_{E_j}=\mathcal I(\vec T_j),
 \qquad
 \nu_P=\mathcal I(\vec T_P),
\]
and therefore
\[
 |\nu_{E_j}-\nu_P|
 =
 |\vec T_j-\vec T_P|.
\]
Thus the oriented unit normals converge uniformly to $\nu_P$
on each compact graphical region.

Since $z_j\to z$, the preceding graphical estimates hold in
a fixed neighborhood of $z$ in the original coordinates.
By translation and dilation, the same argument applies near
every $z\in P\cap B_4$.
A finite covering and local Hausdorff convergence then give
the asserted graphical representations with uniform
$C^{1,\alpha}$ bounds on each compact $K\Subset B_4$, and
\[
 \lim_{j\to\infty}
 \sup_{x\in\operatorname{spt}T_j\cap K}
 |\nu_{E_j}(x)-\nu_P|=0.
\]

\end{proof}
\medskip

We are now ready to   prove Theorem~\ref{thm:main}.

\begin{proof}[Proof of Theorem~\ref{thm:main}]
Fix $1\le n\le7$.
Assume by contradiction that there are integrands $\Phi_j$ and nonaffine entire solutions
$u_j$ such that
\[
 \|\Phi_j-1\|_{C^2(\Sph^n)}\to0.
 \tag{33}\label{eq:sequence-close}
\]
For large $j$, the identity
\[
 D^2\Phi_j(\nu)[\tau,\tau]
 =
 \nabla^2_{\Sph^n}(\Phi_j|_{\Sph^n})(\nu)[\tau,\tau]
 +\Phi_j(\nu)|\tau|^2,
 \qquad \tau\perp\nu,
 \tag{34}\label{eq:spherical-euclidean-hessian}
\]
implies uniform parametric ellipticity and uniform
$C^{1,1}(\Sph^n)$ bounds.

Let $\Sigma_j=\partial E_{u_j}$ be the graph of $u_j$.
By Lemma~\ref{lem:calibration}, each $E_{u_j}$ is a global
$\Phi_j$-minimizer and is directed.
In particular, $\Sigma_j$ is a connected complete
$\Phi_j$-stationary hypersurface.
Fix $p_j\in\Sigma_j$.
By Lemma~\ref{lem:gaussgap}, there exists $q_j\in\Sigma_j$ such that
\begin{equation}\label{eq:fixed-normal-gap}
 |\nu_j(p_j)-\nu_j(q_j)|\ge1.
\end{equation}
Set $d_j:=|p_j-q_j|>0$.
For each $j$, choose $R_j\in SO(n+1)$ such that
$R_j\nu_j(p_j)=e_{n+1}$, and set
\[
 \lambda_j:=(8d_j)^{-1},
 \qquad
 \mathcal S_j(X):=\lambda_jR_j(X-p_j),
 \qquad
 E_j:=\mathcal S_j(E_{u_j}).
\]
Define the transformed integrand by
\[
 \Psi_j(\xi):=\Phi_j(R_j^{-1}\xi).
\]
The transformation $\mathcal S_j$ scales the energy by
$\lambda_j^n$:
\[
 P_{\Psi_j}(\mathcal S_j(E);\mathcal S_j(U))
 =
 \lambda_j^n P_{\Phi_j}(E;U).
\]
Thus $E_j$ is a global $\Psi_j$-minimizer.
Moreover, rotations preserve the spherical $C^2$ bounds, so
\[
 \Psi_j\to|\cdot|
 \quad\text{in }C^2(\Sph^n),
\]
with uniform parametric ellipticity and uniform
$C^{1,1}(\Sph^n)$ bounds.

We also record the transformed calibrations.
Let $\omega_j^0$ be the global closed $C^1$ calibration form
constructed in the proof of Lemma~\ref{lem:calibration}, and set
\[
 T_j:=\partial\llbracketbracket{E_j},
 \qquad
 \omega_j:=
 \lambda_j^n(\mathcal S_j^{-1})^*\omega_j^0.
\]
Then $\omega_j$ is a closed $C^1$ form on $\R^{n+1}$.
Since $R_j$ preserves orientation, it satisfies
\[
 \omega_j(\xi)\le\Psi_j(\nu_\xi)
 \quad\text{for every oriented unit $n$-vector }\xi,
 \qquad
 \omega_j(\vec T_j)=\Psi_j(\nu_{\vec T_j})
 \quad\|T_j\|\text{-a.e.}
\]
Relabel the transformed point $\mathcal S_j(q_j)$ as $q_j$,
and let $\nu_j$ now denote the outer unit normal to $\partial E_j$.
Then
\begin{equation}\label{e4.4}
 0,q_j\in\partial E_j,
 \qquad
 \nu_j(0)=e_{n+1},
 \qquad
 |q_j|=\frac18,
 \qquad
 |\nu_j(0)-\nu_j(q_j)|\ge1.
\end{equation}
By Lemma~\ref{lem:calibration}, $E_j$ is directed in the direction
\[
 a_j:=-R_je_{n+1}.
\]
After passing to a subsequence,
\[
 a_j\to a\in\Sph^n,
 \qquad
 q_j\to q,
 \qquad
 |q|=\frac18.
\]




Applying Lemma~\ref{lem:compactness} with the marked boundary
points $0$, and passing to a further subsequence, we obtain
\[
 \chi_{E_j}\to\chi_E
 \quad\text{in }L^1_{\mathrm{loc}}(\R^{n+1}),
\]
where $E$ is a global Euclidean perimeter minimizer.
For every fixed $\rho>0$ and all sufficiently large $j$,
$
 B_{\rho/2}(q_j)\subset B_\rho(q).
$
The two-sided density estimate \eqref{eq:density} and local
$L^1$ convergence therefore give
\[
 \min\bigl\{
 |E\cap B_\rho(q)|,\,
 |B_\rho(q)\setminus E|
 \bigr\}
 \ge c_0(\rho/2)^{n+1}>0.
\]
The same argument at $0$ gives
\[
 0,q\in\partial E.
 \tag{38}
\]
In particular, $E$ is nontrivial.
By Lemma~\ref{lem:directedclosed}, $E$ is directed in direction $a$.
Proposition~\ref{prop:directed-classification} therefore implies
that $E$ is a halfspace.
Denote it by $H$ and set $P:=\partial H$.
Then $0,q\in P$.
Orient $P$ by the outer normal $\nu_P$ of $H$, so that
\[
 T_P:=\partial\llbracketbracket{H}
\]
is the multiplicity-one oriented boundary current of $H$.


Apply Lemma~\ref{lem:compactness} with $x=0$ and $R=1$.
After passing to a further subsequence, there exists $r\in(1,2)$
such that
\[
 \Per(E_j;B_r)\longrightarrow\Per(H;B_r).
\]
Together with the local $L^1$ convergence, this gives
\[
 T_j\longrightarrow T_P
 \quad\text{as currents in }B_r,
 \qquad
 \|T_j\|\stackrel*\rightharpoonup\|T_P\|
 \quad\text{as Radon measures locally in }B_r.
\]
Set $c:=4/r$ and $D(X):=cX$, and define
\[
 \widetilde E_j:=D(E_j),
 \qquad
 \widetilde T_j:=D_\#T_j,
 \qquad
 \widetilde q_j:=cq_j,
 \qquad
 \widetilde\omega_j:=c^n(D^{-1})^*\omega_j.
\]
The integrands remain $\Psi_j$, and the forms
$\widetilde\omega_j$ are global closed $C^1$ calibrations
for $\widetilde T_j$.
Since $P$ passes through the origin, the rescaled currents satisfy
\[
 \widetilde T_j\longrightarrow T_P
 \quad\text{as currents in }B_4,
 \qquad
 \|\widetilde T_j\|\stackrel*\rightharpoonup\|T_P\|
 \quad\text{as Radon measures locally in }B_4.
\]
Thus all the hypotheses of Lemma~\ref{lem:c1compactness} hold.

 Let $\widetilde\nu_j$ be the outer unit normal to
$\partial\widetilde E_j$.
Lemma~\ref{lem:c1compactness} gives
\[
 \sup_{y\in\partial\widetilde E_j\cap\overline{B_{3/4}}}
 |\widetilde\nu_j(y)-\nu_P|
 \longrightarrow0.
\]
Moreover,
$
 |\widetilde q_j|
 =\frac1{2r}\in\left(\frac14,\frac12\right),
$
and hence both $0$ and $\widetilde q_j$ belong to
$\partial\widetilde E_j\cap\overline{B_{3/4}}$.
Positive dilations preserve unit normals. Hence
\[
 1
 \le |\nu_j(0)-\nu_j(q_j)|
 =|\widetilde\nu_j(0)-\widetilde\nu_j(\widetilde q_j)|
 \le
 |\widetilde\nu_j(0)-\nu_P|
 +
 |\widetilde\nu_j(\widetilde q_j)-\nu_P|
 \longrightarrow0,
\]
contradicting \eqref{e4.4}.
This proves the existence of $\delta_n$ and completes the proof.


\end{proof}

\end{document}